\documentclass[11pt, psamsfonts, reqno]{amsart}
\usepackage{amsmath}
\usepackage{amsthm}
\usepackage{amssymb}
\usepackage{amscd}
\usepackage{amsfonts}
\usepackage{amsbsy}
\usepackage{epsfig}

\newtheorem{theorem}{Theorem}
\newtheorem{prop}[theorem]{Proposition}
\newtheorem{lemma}[theorem]{Lemma}
\newtheorem{corollary}[theorem]{Corollary}
\newtheorem*{claim}{Claim}
\newtheorem{defn}{Definition}

\def\H#1{{\bf #1}}

\def\ie{{\em i.e.,} }
\def\eg{{\em e.g.} }

\newfont\bbf{msbm10 at 12pt}
\def\eps{\varepsilon}
\def\phi{\varphi}
\def\R{{\mathbb R}}

\def\N{{\mathbb N}}

\def\M{{\mathcal M}}

\def\P{{\mathcal P}}

\def\D{{\mathcal D}}

\def\TF{{\mathcal L}}

\def\cont{{\mathcal C}}
\def\es{{\emptyset}}
\def\sm{\setminus}

\def\supp{\mbox{\rm supp}}

\def\crit{\mbox{Crit}}
\def\scrit{\mbox{\tiny Crit}}

\def\bd{\partial }
\def\?{{\bf ???}}

\def\le{\leqslant}
\def\ge{\geqslant}

\newcommand{\st}{such that }

\newcommand{\bcl}{Borel-Cantelli Lemma }

\def\hmu{\hat{\mu}}
\def\hf{\hat{f}}

\def\hy{\hat{y}}

\def\hU{\hat{U}}

\def\c{Z}
\def\A{{\mathcal A}}
\def\H{\mathcal H}

\begin{document}

\title[Return time statistics of interval maps]{Return time statistics of invariant measures for interval maps with positive Lyapunov exponent}
\author{Henk Bruin and Mike Todd}
\thanks{
This research was supported by EPSRC grant GR/S91147/01.
MT was partially supported by FCT grant SFRH/BPD/26521/2006 and CMUP}
\subjclass[2000]{37E05, 37A05, 37B20, 37D25}
\keywords{return time statistics, interval maps, non-uniform hyperbolicity, equilibrium states, Gibbs property}
\maketitle

\begin{abstract}
We prove that multimodal maps with an absolutely continuous invariant measure have exponential return time statistics around almost every point.  We also show a `polynomial Gibbs property' for these systems, and that the convergence to the entropy in the Ornstein-Weiss formula has normal fluctuations.  These results are also proved for equilibrium states of some H\"older potentials.
\end{abstract}

\section{Introduction}
Return time statistics refers to the distribution of return times
to (usually small) sets $U$ in the phase space of a measure
preserving dynamical system. There have been various approaches to
estimate these distributions in the literature. The earlier
methods pertain to hyperbolic dynamical systems (such as Markov chains \cite{Pit}, and Anosov
diffeomorphisms \cite{Hir}) as these benefit most directly
from the techniques of i.i.d.
stochastic processes, the area in which return time statistics was
studied first.
Further results for interval maps have been obtained in \eg \cite{CG,CGS}.
Gradually methods were developed to treat
non-uniformly hyperbolic systems,  and in \cite{BSTV} it was pointed
out that the return time statistics of a dynamical system coincides
with the return time statistics of a (first return) induced map. If
this first return map itself is hyperbolic, then the above theory can be applied immediately, but the existence of a hyperbolic first return map is
a serious restriction on general dynamical systems, especially
when (recurrent) critical points are present.

In \cite{BrV} this problem was overcome in the context of unimodal
interval maps satisfying a summability condition on the
derivatives along the critical orbit. Instead of a first return
map, a hyperbolic inducing scheme was used, where the inducing
time is a suitable, rather than a first, return to a specific
subset $Y$ of the interval.  The method was to use the so-called `Hofbauer tower' see \cite{kellift,BrCMP}, on which the inducing scheme corresponds to a first return map to a suitable subset $\hat Y$ of the Hofbauer tower.

The properties of the density of the absolutely continuous invariant measure (acip), which are well understood for a map satisfying a summability condition on the
derivatives along the critical orbit, were used extensively in \cite{BrV}.  However, as can be seen in \cite{BR-LSS}, acips are known to exist even when such summability conditions do not hold, and even when the map is multimodal.  In this paper we show that such properties on the density are not required.  This allows us to significantly improve on the class of maps and measures we can deal with.  Here:
\begin{itemize}
\item $f$ can be any non-flat $C^3$ multimodal map,
\item $\mu$ can be an arbitrary invariant probability measure with
positive Lyapunov exponent: $\lambda(\mu) := \int \log |Df| \ d\mu > 0$.
\end{itemize}
In these cases, exponential return time statistics to balls is obtained for any acip, as well as for equilibrium states of a natural class of potentials, $x\mapsto -\delta\log|Df(x)|$ for $\delta$ close to 1.  In addition, we obtain a `polynomial Gibbs property'
and fluctuation results in the Ornstein-Weiss formula, for both acips and equilibrium states of certain H\"older potentials,
provided a very weak growth condition of derivatives along
critical orbits is satisfied.

Let us start by introducing the concept of return time statistics
in more detail. Let $(I,f,\mu)$ be a measure preserving ergodic
dynamical system.  For a measurable set $U_z \subset I$ containing
some $z\in I$, let $\mu_{U_z} = \frac{1}{\mu(U_z)} \mu|_{U_z}$ be
the conditional measure on $U_z$ and $r_{U_z}(x)$ be the first
return time of a point $x\in U_z$ to $U_z$. Whenever $z$ is not a
periodic point, the return time $r_{U_z}(x) \to \infty$ as
$\mu(U_z) \to 0$, but Kac's Lemma states that $\int_{U_z}
r_{U_z}(x) d\mu_{U_z} = 1$. Therefore, when $r_{U_z}$ is
scaled by $\mu(U_z)$, we can hope for a well-defined
distribution $G:[0,\infty) \to \R$ such that, for $t \in [0,\infty)$
$$
\mu_{U_z}\left( x\in U_z\ :\ r_{U_z}(x)\mu(U_z)>t \right) \to G(t)
$$
as $\mu(U_z) \to 0$.  We refer to this as
the \emph{return time statistics of $(f, \mu)$}.  For many mixing systems it is known that the return time statistics are exponential, \ie $G(t) = e^{-t}$,
see \eg \cite{Abadi, Coelho} for various results on some well behaved systems.
This is what we find for systems considered in the latter sections of this paper.  For multiple return time statistics of these cases we expect to find Poissonian laws, see \cite{HSV}.

The natural choice for the sets $U_z$ are balls or cylinder sets, but results on balls are in general harder to prove because of the lack of (H\"older) regularity of indicator functions $\chi_{U_n}$.
Also the Gibbs property only gives information on cylinder sets.
Therefore, in dimension greater than $1$, most results known pertain to cylinder sets, and not (yet) to balls. See \cite{Sau} for more information on this issue.

However, the literature contains examples of behaviour far from exponential in different settings.  For example Coelho and de Faria \cite{CoedeF}
find examples of continuous and discontinuous distributions other than exponential,  for circle diffeomorphisms, see \cite{DM} for further results
in this direction.
Moreover, if we do not assume that the sequence of shrinking sets $U_n$ are
balls/cylinders then any continuous distribution can be
obtained, see Lacroix \cite{Lac}.
In fact, Lacroix also shows that any possible return statistics can be
achieved for cylinders for well-chosen Toeplitz flows.
Thus it is important to emphasise that in this paper we will focus on return time statistics to
balls for non-uniformly expanding interval maps.

We next explain the result of \cite{BSTV} which allows us go to from return time statistics of a first return map to the return time statistics of the original system.
Consider an open set $Y\subset X$ and let $R_Y : Y\rightarrow Y$
be the first return map. As above, we denote the conditional measure on $Y$ by $\mu_Y$, which must be $R_Y$-invariant and ergodic.  For $z\in Y$ and $\alpha>0$, let $U_\alpha=U_\alpha(z)$ be the $\alpha$-ball around $z$.  Let $r_{U_\alpha}(x)$ (resp. $r_{R_Y,U_\alpha}(x)$) be the first return time into $U_\alpha$ for $f$ (resp. $R_Y$). We suppose that $(Y, R_Y,\mu_Y)$ has return time statistics $G(t)$, \ie for $\mu_Y$-a.e. $z\in Y$, there exists $\eps_z(n) \ge 0$ with $\eps_z(n)\to 0$ as $\alpha \to 0$ such that
\begin{equation}\label{eq:return}
\sup_{t\ge0}\left|\mu_{U_\alpha}\left(x\in U_\alpha \ : \ r_{R_Y,U_\alpha}(x) >
\frac{t}{\mu(U_\alpha)}\right) - G(t)\right| <\eps_z(n).
\end{equation}

The key result of \cite{BSTV} is that $(Y, R_Y, \mu_Y)$ enjoys the same distribution as $(I, f,\mu)$:
\begin{theorem}\label{thm:ret map same ret stats}
Suppose that the function $G$ in \eqref{eq:return}
is continuous on $[0,\infty)$.  Then for $\mu$-a.e. $z\in Y$, there exists $\delta_{z}(\alpha) > 0$ with $\delta_{z}(\alpha)\rightarrow 0$ as $\alpha\to 0$ such that:
\[
\left|\mu_{U_\alpha}\left(x\in U_\alpha \ : \ r_{U_\alpha}(x) >
\frac{t}{\mu(U_\alpha)}\right) - G(t)\right| < \delta_{z}(\alpha).
\]
\end{theorem}

Note that the theorem can also be applied to cylinders rather than balls.

This theorem requires a first return map, rather than an arbitrary induced map, and we will use the Hofbauer tower
to bridge that gap. The requirement that $\mu$ has
a positive Lyapunov exponent is needed to `lift'
$\mu$ to this Hofbauer tower.
\emph{Liftability} is an abstract convergence property
(in the vague topology) of Ces\`aro means of a measure $\mu$ imposed
on the Hofbauer tower. It was introduced by Keller
\cite{kellift}. He showed in the context of one-dimensional  maps,
that $\mu$ having positive entropy ($h_{\mu} > 0$) or positive
Lyapunov exponent
both imply liftability. (In fact, for non-atomic measures,
$\lambda(\mu) > 0$ is  equivalent to liftability, see \cite{BrK}.)

Let us now explain which type of
induced systems we will consider.
We fix $\delta>0$ and some interval $Y$.
We say that the interval $Y'$ is a \emph{$\delta$-scaled neighbourhood of $Y$} if, denoting the left and right components of $Y' \sm Y$ by $L$ and $R$ respectively, we have $|L|,|R|=\delta|Y|$.
Next define an inducing scheme $(Y,F)$ as follows.
Let $Y'$ be a $\delta$-scaled neighbourhood of $Y$ and define $\tau_{Y,\delta}(y)$ to be
\[
\min \left\{ i\ge 1 : f^i(y)\in Y \hbox{ and $\exists H \ni y$ with } f^i|_H:H \to Y'\hbox{  homeomorphic} \right\}.
\]
We call this the \emph{first $\delta$-extendible return time to $Y$}.  For $y\in Y$ we let $F(y):=f^{\tau_{Y,\delta}(y)}(y)$.  Given a point $z\in I$ we will take a sequence of nested intervals $\{J_n\}_n$ \st $\bigcap_n J_n = \{z\}$, we will denote
\begin{equation}\label{eq:Fn}
F_n = f^{\tau_{J_n}}:J_n \to J_n
\end{equation}
to be the first $\delta$-extendible return map to $J_n$,
with first $\delta$-extendible return time $\tau_{J_n}=\tau_{J_n,\delta}$.
We explain in Section~\ref{sec:lifting} how the intervals $J_n$ are chosen.
Associated to an $f$-invariant measure of positive Lyapunov exponent,
there is an $F_n$-invariant measure $\mu_{F_n}$ for each there $n$
such that
\[
\mu(A)  = \frac{1}{\int \tau_{J_n}~d\mu_{F_n}}
\sum_i \sum_{k=0}^{i-1}
\mu_{F_n}\left(f^{-k}(A) \cap \{ \tau_{J_n} = i\}\right),
\]
see \eqref{eq:muA}.

We denote the, finite, set of critical points by $\crit$.
We say that $c\in \crit$ is {\em non-flat} if there exists a diffeomorphism $g_c:\R \to \R$ with $g_c(0)=0$ and $1<\ell_c<\infty$ \st for $x$ close to $c$, $f(x)=f(c)\pm|\phi_c(x-c)|^{\ell_c}$.  The value of $\ell_c$ is known as the \emph{critical order} of $c$.
We write $\ell_{\max} := \max_{c \in \scrit} \ell_c$.
Let
\[
NF^k := \left\{ f:I \to I : f \mbox{ is } C^k, \mbox{ each $c \in \crit$
is non-flat and } \inf_{f^n(p) = p} |Df^n(p)| > 1 \right\}.
\]
Maps in $NF^2$ have no wandering intervals, see \cite{MSbook}, and therefore
$\sup_x |\c_n[x]| \to 0$ as $n \to \infty$.
By \cite{SV}, if $f\in NF^3$ we can use the Koebe Lemma
(see \cite{MSbook}) to say that the
first $\delta$-extendible return map $F$ has bounded distortion.
For some of the results below we need an {\em expansion} condition
on critical orbits.
Therefore we can use results from \cite{BR-LSS}
(namely Main Theorem' and Theorem 1 respectively)
which state that a map $f \in NF^3$ with
$\min_{c \in \scrit} \liminf_n |Df(f(c))| \ge L$ has an acip, and also satisfies
a backward contraction property called $BC(2)$.
The number $L$ depends only on the
cardinality of the critical set and the maximal critical order $\ell_{\max}$
of $f$. With this in mind we define
\[
NF_+^k  := \left\{ f \in NF^k : \
\min_{c \in \scrit} \liminf_n |Df(f(c))| \ge L(\ \# \crit(f), \ell_{\max}(f)\ ) \right\}.
\]
Any map in this class cannot be infinitely renormalisable.

The following is our first main theorem.  This theorem also holds for return time statistics to cylinders.

\begin{theorem}
Let $f\in NF^3$ and $(I,f, \mu)$ be liftable.  Suppose that for $\mu$-a.e. $z\in I$ there exists $\delta>0$ and a nested sequence of intervals
$\{J_n\}_n$ \st
$\cap_nJ_n=\{z\}$ and for all $n$, the system $(J_n, F_n, \mu_{F_n})$
from \eqref{eq:Fn} has return time statistics given by a continuous function
$G:[0,\infty) \to [0,1]$.
Then $(I,f, \mu)$ also has return time statistics given by $G$.
\label{thm:induced same as original}
\end{theorem}

We will apply this theorem to a class of equilibrium states, which includes acips. We say that $\mu$ is an \emph{equilibrium state} of the \emph{potential} $\phi:I \to \R$ if
its \emph{free energy} $h_\mu + \int \phi d\mu$ is equal to the \emph{pressure}
\[
P(\phi) := \sup_{\nu \in \M_{erg}}
\left\{h_\nu + \int \phi~d\nu:\int\phi~d\mu<\infty\right\},
\]
where $\M_{erg}$ denotes the set of all ergodic invariant
probability measures.  We say that $m_\phi$ is a \emph{conformal measure} for $\phi$ if for all Borel sets $A\subset I$ with $f:A \to f(A)$ bijective,
$$
m_\phi(f(A))=\int_A e^{-\phi}~dm_\phi.
$$
We use the abbreviation $m_\delta$ for the conformal measure for the potential
$\phi_\delta:x\mapsto -\delta\log|Df(x)|$.
For our first application of Theorem~\ref{thm:induced same as original}, we will be interested in potentials of the form $\phi_\delta:x\mapsto -\delta\log|Df(x)|$.
For the specific choice $\phi_1 = -\log|Df|$, any equilibrium state must be an acip, see \cite{Led, Ruelle}.

\begin{theorem}
Suppose that $f\in NF^3$ and assume that the equilibrium state
$\mu_\delta$ and conformal measure $m_{\delta}$
exist for $\delta\in [0,1]$, and $\mu_\delta \ll m_{\delta}$.
Then $(I,f,\mu_\delta)$ has exponential return time statistics
(\ie $G(t) = e^{-t}$) to balls around $\mu_\delta$-a.e. point.
\label{thm:exp stats}
\end{theorem}

Note that in the case $\delta=1$, the conformal measure is Lebesgue and so all that is required is the existence of an acip.
The proof of Theorem~\ref{thm:exp stats} implies the following corollary, which specifies cases where the equilibrium state and (for the induced system)
the conformal measures are known to exist.  For the existence results in parts (2) and (3), see \cite{BTeqnat}.

\begin{corollary}
Suppose that $f\in NF_+^3$.
\begin{enumerate}
\item  On each transitive component of $(I,f)$, there is an acip $\mu_1$, and the system $(I,f,\mu_1)$ has exponential return time statistics.
\item Suppose that for some $\delta_0 \in (0,1)$, $C > 0$ and $\beta > \ell_{max}(1+\frac1{\delta_0})
- 1$,
\[
|Df^n(f(c))| \ge C n^\beta \quad \mbox{ for all } c \in \crit
\mbox{ and } n \ge 1.
\]
Then there exists $\delta_1 \in (\delta_0,1)$ \st for all $\delta\in (\delta_0,1)$, on each transitive component of $(I,f)$ there exists a unique equilibrium state $\mu_\delta$ for the potential $\phi_\delta$, and $(I,f,\mu_\delta)$ has exponential return time statistics.
\item Suppose that there exist $C,
\alpha > 0$ \st
\[
|Df^n(f(c))| \ge C e^{\alpha n} \mbox{ for all } c \in \crit \mbox{
and } n \in \N.
\]
Then there exist $\delta_1<1<\delta_2$ \st on each transitive component of $(I,f)$, there is a unique equilibrium state $\mu_{\delta}$  for the potential $\phi_\delta$, and $(I,f,\mu_\delta)$ has exponential return time statistics.\footnote{In \cite{PeS}, the existence of an equilibrium state $\mu_\delta$
was shown for all $\delta \in (-\eps,1+\eps)$
for a class of logistic maps near the Chebyshev polynomial $f(x) = 4x(1-x)$, where $\eps>0$.}
\end{enumerate}
\label{cor:applications}
\end{corollary}

We also
consider the class of potentials
$$
\H := \{ \phi:I \to \R : \phi \mbox{ is  H\"older and }
\sup\phi-\inf\phi<h_{top} \},
$$
where $h_{top}$ is the topological entropy of $f$.
Keller proved in \cite{kellholder} that if $f$ is piecewise monotone
(\ie $f$ has finitely many continuous monotone branches, but discontinuities between branches are allowed) and $\phi\in \H$, then on each transitive component of $(I,f)$ there is a unique equilibrium state $\mu$ for $(I,f,\phi)$.
See also \cite{BTeqgen} where a similar result was proved, with weaker conditions on $\phi$, but stronger conditions on $f$.  The following proposition gives the return time statistics to balls for these measures.  For a similar result, but for cylinders, see Paccaut \cite{Pac}.

\begin{prop}
Suppose that $f$ is piecewise monotone, and $\phi\in \H$ is a potential.
Then for every equilibrium state $\mu$ for this potential,
$(I,f,\mu)$ has exponential return
time statistics to balls around $\mu$-a.e.\ point.
\label{prop:exp stats cyl}
\end{prop}

In the setting of the above proposition, Keller proved exponential decay of correlations for the original system $(I, f,\mu)$ for a class of observables which includes characteristic functions on balls.  Therefore, in contrast to our results for acips, using the ideas of \cite{BSTV} we can prove this result directly, with no inducing.

Our next result concerns a weak version of the Gibbs property.
Let $\P_1$ be the partition of $I$ into maximal (closed) intervals
such that $f:Z \to f(Z)$ is a homeomorphism for each $Z \in \P_1$.
Refine the partition $\P_n = \bigvee_{i=0}^{n-1} f^{-i} \P_1$ and by convention let $\P_0 = \{I\}$.  We refer to the elements of $\P_n$ as cylinder sets, and we write $\c_n[x]$ to indicate the
cylinder set in $\P_n$ containing $x$.  If $x\in \bd \c_n$ then $\c_n[x]$ is not unique, but this applies only to countably many points.

Let $S_n \phi(x) := \sum_{k=0}^{n-1} \phi \circ f^k(x)$ be the $n$-th ergodic
sum along the orbit of $x$.
We say that $\mu$ satisfies the \emph{polynomial Gibbs property}
with exponent $\kappa$ if for $\mu$-a.e. $x$, there is $n_0=n_0(x)$ such that
\begin{equation}\label{eq:polGibbs}
\frac{1}{n^\kappa} \le \frac{\mu(\c_n[x])}{e^{S_n \phi(x)-nP(\phi)}}
\le n^\kappa,
\end{equation}
for all $n \ge n_0$.  If $\mu$ is an acip, and hence an equilibrium state for the potential
$\phi = -\log |Df|$, then the pressure $P(\phi) = 0$ and the quantity to estimate in \eqref{eq:polGibbs} simplifies to $\mu(\c_n[x]) |Df^n(x)|$.  Formula \eqref{eq:polGibbs} was used in \cite{BrV} and can be compared with the `weak Gibbs property' given by Yuri \cite{Yuri},
for which the Gibbs constants depend only on $n$, and the `non-lacunary measures' of \cite{OVeq} where the constants depend on $x$ and $n$, but can grow at any subexponential rate.

\begin{theorem}
For any $f\in NF_+^3$, the following hold:
\begin{itemize}
\item[(a)] For each transitive component of $(I,f)$, there is a unique acip $\mu$ and $\mu$ is polynomially Gibbs.
More precisely, if
$\gamma>4\ell_{max}^2$ and $\gamma' > 2$, then for $\mu$-a.e. $x$ there exists
$n_0$ \st for $n\ge n_0$,
$$
\frac1{n^{2\gamma}} \le \frac{|f^n(\c_n[x])|}{n^{\gamma}}\le
\mu(\c_n[x])|Df^n(x)|\le n^{\gamma'}.
$$
\item[(b)]
For any potential $\phi\in \H$, for each transitive component of $(I,f)$ there exists a unique equilibrium state $\mu$, which is polynomially Gibbs.
\end{itemize}
\label{thm:weak G}
\end{theorem}

Note that in part (b) we ask for stronger conditions on $f$ than we did in Proposition~\ref{prop:exp stats cyl}. Under these conditions, we proved in \cite{BTeqgen} that $\mu$ is compatible to an inducing scheme with `exponential tails', which allows us to prove the above result.
The precise exponent $\kappa$ of the polynomial Gibbs property in
condition (b) is given in the proof of Proposition~\ref{prop:poly Gibbs}.
This depends on the rate of decay of the tails for $\mu$.

The final results of this paper concern the normal fluctuation in the Ornstein-Weiss formula of return times.  The Ornstein-Weiss formula says in this context that the first return time to $\c_n[x]\in \P_n$ satisfies
\begin{equation}\label{eq:OW}
\lim_{n \to \infty} \frac1n \log r_{\c_n[x]}(x) = h_\mu
\mbox{ for } \mu\mbox{-a.e.} \ x.
\end{equation}
If $\mu$ is an invariant probability measure, then the variance
$\sigma_\mu^2$ of the process $\{ \phi \circ f^n\}_{n \ge 0}$ is defined by
$$
\sigma_\mu^2=\sigma_\mu(\phi)^2:=\int\phi^2~d\mu-
\left(\int\phi~d\mu\right)^2
+2\sum_{n=1}^\infty\left[\int\phi\circ f^n\cdot
\phi~d\mu -\left(\int\phi~d\mu\right)^2\right],
$$
where in case of the acip, $\phi:=-\log|Df|$.
We have  $\sigma_\mu>0$, except when $\phi$ is a coboundary, \ie
$\phi = \psi  \circ f - \psi$ for some measurable function $\psi$.
Potentials are unlikely to have zero variance.
For example for $\phi = -\log|Df|$ and $f(x) = ax(1-x)$, the only parameter for which $\phi$ is a coboundary is believed to be $a = 4$, cf. Corollary 3 in \cite{BHN}. This is a special case of the broader notion of Liv\v{s}ic regularity.

\begin{theorem}
Let $f \in NF_+^3$ and assume that one of the following conditions holds.
\begin{itemize}
\item[(a)]
For some $\beta>4\ell_{max}-3$,
\begin{equation}\label{eq:polynomial}
|Df^n(f(c))| \ge C n^\beta \quad \mbox{ for all } c \in \crit
\mbox{ and } n \ge 1,
\end{equation}
and $\mu$ is the acip.
\item[(b)] The potential $\phi\in \H$ and $\mu$ is the equilibrium state for $\phi$.
\end{itemize}
If $\sigma^2_\mu > 0$, then
\begin{equation*}
\mu\left(
x\in X:\frac{\log r_{\c_n(x)}(x)-nh_\mu}{\sigma_\mu\sqrt n}>u\right)
\to \frac1{\sqrt{2\pi}}\int_u^\infty e^{-\frac{x^2}2}~dx.
\label{eq:fluc}
\end{equation*}
\label{thm:fluc}
\end{theorem}
For condition (b), see Paccaut \cite{Pac}, where a similar result is proved for another class of equilibrium states.

This paper is organised as follows. In Section~\ref{sec:lifting} we
give the basic definitions for interval maps, and we discuss the Hofbauer tower and its lifting properties.
Theorem~\ref{thm:ret map same ret stats} is proved in Section~\ref{sec:ind stats same}.
In Section~\ref{sec:exp stats} we focus on the exponential return
time statistics of acips and equilibrium states of potentials in $\H$.
Next, in Section~\ref{sec:W Gibbs} we present our results on the polynomial
Gibbs property.  The fluctuation results for the Ornstein-Weiss formula (Theorem~\ref{thm:fluc}) is given in Section~\ref{sec:fluc}.

Throughout calculations, $C$ will be a constant depending only on the map
$f$.

{\em Acknowledgement:} We would like to thank the referee for helpful comments.

\section{Lifting Measures to the Hofbauer Tower}
\label{sec:lifting}

The \emph{Hofbauer tower} (or \emph{canonical Markov extension}) is defined as
the quotient space
\[
\hat I := \left(\bigsqcup_{n \ge 0}\ \bigsqcup_{\c_n \in \P_n}
f^n(\c_n)\right) { /} \sim
\]
where $f^n(\c_n) \sim f^k(\c_k)$ if $f^n(\c_n) = f^k(\c_k)$.  We
denote the domains of $\hat I$ by $D=D(\c_n)$ and the collection of
all such domains by $\D$.
Points in $\hat I$ are written as $\hat x = (x,D)$, where
$D = D_{\hat x}$ is element of $\D$ containing $\hat x$.

We write $D \to D'$ for $D, D'\in\D$ if there exist $\c_n\in\P_n$
and $\c_{n+1}\in\P_{n+1}$ such that $\c_{n+1}\subset \c_n$, $D =
D(\c_n)$ and $D' = D(\c_{n+1})$. This gives $\D$ a graph
structure with domains $D$ as vertices.  For each $D=D(Z_n) \in \D$ has
at least one and at most $\,\#\P_1$ outgoing arrows.
The map $\hat f:\hat I \to \hat I$ is defined as
\[
\hat f(x,D) =  (f(x),D'),
\]
where $D' = f^{n+1}(\c_{n+1})$ for that particular element $\c_{n+1}
\in \P_{n+1}$ such that $x \in f^n(\c_{n+1})$ and $D \to D'$. Again
$\hat f(x,D)$ is uniquely defined for $x \not\in f^n(\bd
\c_{n+1})$; otherwise $\hat f$ is multivalued at $\hat x=(x,D)$. By
definition we have the following property: The system $(\hat I,
\hat f)$ is a Markov map with Markov partition $\D$. The \emph{
natural projection } $\pi:\hat I \to I$ is the (countable to
one) inclusion map from $\hat I$ to $I$, and
\[
\pi \circ \hat f = f \circ \pi.
\]
Let $i$ be the trivial bijection mapping (inclusion) $I$ to $\hat
I_0$ (note that $i^{-1}=\pi|_{\hat I_0}$) and let
$\hat\mu_0:=\mu\circ i^{-1}$ and
\begin{equation}
\hat{\mu}_n := \frac{1}{n}\sum_{k=0}^{n-1} \hat\mu_0
\circ\hat{f}^{-k}. \label{eqn:mulift}
\end{equation}
We wish to find some limit $\hat \mu$ of a subsequence of
$\{\hat\mu_n\}_n$.

Note that, as $\hat I$ is generally noncompact, the sequence $\{
\hat\mu_n \}$ may not have a convergent subsequence in the weak
topology. Instead we use the vague topology (see \eg \cite{bil}):
Given a topological space, a sequence of measures $\sigma_n$ is
said to converge to a measure $\sigma$ in the vague topology if
for any function $\phi \in \cont_0(\hat I)$ (where $\cont_0(\hat I)$ is the
set of continuous functions with compact support in $\hat I$), we have
$\lim\limits_{n \to \infty} \sigma_n(\phi) = \sigma(\phi)$.

A measure $\mu$ on $I$ is \emph{liftable} if a vague limit $\hat
\mu$ obtained in \eqref{eqn:mulift} is not identically $0$.
We define the \emph{Lyapunov exponent} of $\mu$ to be $\int\log|Df|~d\mu$.
In
the following theorem we provide assumptions which ensure
$\hat\mu\not\equiv 0$.

\begin{theorem}
Any ergodic invariant measure with positive Lyapunov exponent for
a $C^1$ interval map is liftable to a measure $\hat\mu$ where
$\hat\mu\circ\pi^{-1} = \mu$. \label{thm:lifting pos lyap}
\end{theorem}

\begin{proof}
For the proof of this see \cite{kellift}, see also \cite{BrK}.
\end{proof}

\section{Return Statistics via the Hofbauer tower}
\label{sec:ind stats same}

In \cite{BSTV} it was shown that dynamical systems $(X,f,\mu)$ and
first return maps $(Y,F,\mu_Y)$ to fixed subsets $Y \subset X$
have the same return time statistics. If $(Y,F)$ is hyperbolic, then it
is commonly expected (and in many cases proved) that return
time statistics will be exponential for balls, or at least for cylinder
sets. However, typically no hyperbolic return maps can be found on
sets with $\mu(Y) > 0$. The idea from \cite{BrV}, which we will
extend here, is that there frequently are sets $Y$ with induced
(rather than first return) maps $F$ such that $Y$ can be lifted to
a set $\hat Y \subset \pi^{-1}(Y)$ in the Hofbauer tower such
that $F$ lifts to a first return map on $\hat Y$. As the set $Y$
decreases in size, $F$ will be closer to the true first return
map, and hence we can approximate the return time statistics on the
original system.  That is, we prove Theorem~\ref{thm:induced same
as original}.

We first explain the inducing schemes we consider.
An inducing scheme $(Y,F,\tau)$ for $Y \subset I$ is a
generalisation of a first return map.
It consists of a collection $\{ Y_i \}_i$ such that
$F|_{Y_i} = f^{\tau_i}|_{Y_i}:Y_i \to Y$ is monotone onto for
some $\tau_i \in \{1,2,\dots\}$. The function $\tau:\cup_i Y_i \to \N$
with $\tau(x) = \tau_i$ if $x \in Y_i$ is called the \emph{inducing time}.
It is well-known that if $\mu_F$ is an $F$-invariant measure, with
\begin{equation*}
\Lambda := \sum_i \tau_i \, \mu_F(Y_i) < \infty,
\end{equation*}
then $\mu$ defined by
\begin{equation}\label{eq:muA}
\mu(A) = \frac1{\Lambda} \sum_i \sum_{k=0}^{\tau_i-1}
\mu_F(f^{-k}(A) \cap Y_i)
\end{equation}
is $f$-invariant.

We next explain the relation between a first extendible return map and a first return map on the Hofbauer tower.  We fix $\delta>0$ and let $z$ be a typical point of $\mu$.  Let $J_n:=\c_n[z]$ and $I_n$ be a $\delta$-neighbourhood of $J_n$.
Let $U$ be any open interval such that $J_n \supset U \owns z$.
Let $\hat R_U:\pi^{-1}(U) \to \pi^{-1}(U)$ be the first return map to $\pi^{-1}(U)$ by $\hat f$, and denote the return time function by $r_U$.
Define $\hat I_n \subset \pi^{-1}(I_n)$ to be the maximal set \st
$\hat I_n \cap D \neq \es$ for $D \in \D$ implies that
$\pi^{-1}(I_n)\cap D$ is compactly contained in $D$. Now let
$\hat J_n:=\pi^{-1}(J_n)\cap \hat I_n$ and denote the first return map
by $\hat f$ to $\hat J_n$ by $R_{\hat J_n}$.  Note that $R_{\hat J_n}$ is extendible to $\hat I_n$.
Define $\tilde F_n(y) := \pi \circ R_{\hat J_n} \circ \pi|_{\hat
J_n}^{-1}(y)$.   \cite{BrCMP} implies that $\tilde F_n$ is well defined.  As in the introduction, we consider $\tau_{J_n}=\tau_{J_n,\delta}$.

\begin{lemma}
For the first return time $r_{\hat J_n}$ to $\hat J_n$ we have
$\tau_{J_n} = r_{\hat J_n}\circ\pi_{\hat J_n}^{-1}$.
\label{lem:ret induced same}
\end{lemma}

This implies that $\tilde F_n$ is the same as $F_n$ defined by \eqref{eq:Fn}.

\begin{proof} This was shown in \cite[Lemma 2]{BrCMP},
\end{proof}

We say that $r_U$ is \emph{$(n,\delta)$-extendible at $x$} if $f^{r_U(x)}$ can be extended homeomorphically locally around $x$ to $I_n$.

\begin{lemma}
For any $z$ as above we have
$$\lim_{n \to \infty}\sup_{z \in U\subset J_n}
\mu_U\left( x \in U: r_U(x) \hbox{ is not } (n,\delta)\hbox{-extendible at } x\right)
=0,$$
\label{lem:extendible}
where the supremum is taken over intervals $U$.
\end{lemma}

\begin{proof}
Let $\hat U_n:=\pi^{-1}(U)\cap \hat I_n$. By Theorem~\ref{thm:lifting pos lyap}, the construction of
$\hat R_U$ and Lemma~\ref{lem:ret induced same}, we have
\[
\mu\left(z \in U: r_U \hbox{ is not } (n,\delta)\hbox{-extendible}\right) =
\hat\mu\left(\hat R_U^{-1}(\pi^{-1}(U) \cap \hat U_n)\right) =
\hat\mu(\pi^{-1}(U) \sm \hat U_n)
\]
by the $\hat R_U$-invariance of $\hat\mu$.  To prove our lemma we
must estimate this quantity relative to $\mu(U)$.

Fix  $0<\eps<1$, and let $D \in \D$ be any domain in the Hofbauer tower. Let $\rho_D(z)$ be given by
$\frac{d\hat\mu\circ \pi|_{D}^{-1}}{d\mu } (z)$; obviously
$\rho_D(z) = 0$ if $z \notin D$. Recalling from
Theorem~\ref{thm:lifting pos lyap} that $\mu =\hat\mu \circ
\pi^{-1}$, we have that $\sum_{D \in \D}\rho_D(z)=1$ for
$\mu$-a.e. $z$. Clearly, there exists some finite subcollection
$\D'$ of $\D$ \st $\sum_{D \in \D'} \rho_D(z) \ge (1-\eps)$.
For each $D$ we say that $(\ast)_D$ holds for $n$ if
\begin{enumerate}
\item  $\pi^{-1}(I_{n}) \cap D$ compactly contained
in $D$; and
\item for any $U\subset J_{n}$, $\frac{\hat\mu(\pi^{-1}(U)
\cap D)}{\mu(U)} \ge (1-\eps)\rho_D(z). $
\end{enumerate}
The first condition trivially holds for any large $n$.  We claim
that the second condition holds for a.e. $z$, when $n$ is
sufficiently large.  To prove this claim, note that we have $0\le
\rho_D\le 1$.  We divide $[0,1]$ into pieces $\{\eta_i\}_i$ of
size $\frac\eps2$.  Choose $\beta_i:=\rho_D^{-1}(\eta_i)$ so that  $z$ is a density point of $\beta_i$.  Note that for $y\in \beta_i$,
$|\rho_D(y)-\rho_D(z)|\le \frac\eps2$.  Then we claim that for
$U\ni z$ a small enough neighbourhood of $z$, we have $$\frac{\hat\mu\circ\pi|_{D}^{-1}(U)}{\mu(U)}
\ge (1-\eps)\rho_D(z).$$

To prove the claim, we have \begin{align*}
\frac{\hat\mu\circ \pi|_{D}^{-1}(U)}{\mu(U)} & =
\frac1{\mu(U)}\int_U\rho_D~d\mu =
\frac1{\mu(U)}\left(\int_{U\cap\beta_i}\rho_D~d\mu+
\int_{U\sm\beta_i}\rho_D~d\mu\right)\\
& \ge \frac{\mu(U\cap\beta_i)}{\mu(U)}
\left(1-\frac\eps2\right)\rho_D(z) - \frac{\mu(U\sm\beta_i)}{\mu(U)}.
\end{align*}  Since $z$ is a density point of $\beta_i$, we have
$$\frac{\mu(U\cap\beta_i)}{\mu(U)} \to 1 \hbox{ and }
\frac{\mu(U\sm\beta_i)}{\mu(U)} \to 0$$ as $U\to z$.  Thus for
large enough $n$, the second condition must hold for $z$.

There exists $N$ \st $(\ast)_D$ holds for all $n \ge N$ and
$D\in \D'$. Therefore, if $n \ge N$ then
\begin{align*}
\frac{\hat\mu(\pi^{-1}(U) \cap \hat U_{n})}{\mu(U)} &= \sum_{D \in
\D} \frac{\hat\mu(\pi^{-1}(U) \cap \hat U_{n} \cap
D)}{\mu(U)} \\
&= \sum_{D \in \D'} \frac{\hat\mu(\pi^{-1}(U) \cap \hat U_{n} \cap
D)}{\mu(U)} + \sum_{D \in \D \setminus \D'}
\frac{\hat\mu(\pi^{-1}(U)
\cap \hat U_{n}\cap D)}{\mu(U)}\\
&\ge (1-\eps)\sum_{D \in \D'}\rho_D(z) \ge (1-\eps)^2.
\end{align*}
Therefore, for all $n\ge N$,
\[
\frac{\hat\mu(\pi^{-1}(U) \sm \hat U_n)}{\mu(U)} \le 1-(1-\eps)^2
< 2\eps.
\]
As $\eps>0$ can be taken arbitrarily small, the proof is
complete.
\end{proof}

\begin{proof}[Proof of Theorem~\ref{thm:induced same as original}]
Let $\alpha_n = \sup_{z \in U \subset J_n}
\frac{\hmu(\hU_n)}{\mu(U)}$. As we have seen in
Lemma~\ref{lem:extendible}, $\lim_{n\to \infty} \alpha_n = 1$. Because
$f \circ \pi = \pi \circ \hf$ we have
\[
\mu_U\left( y\ : \ r_U(y) > \frac{t}{\mu(U)}\right) =
\hat\mu_{ \pi^{-1}(U) } \left( \hy \ : \
r_{\pi^{-1}(U)}(\hy) > \frac{t}{\mu(U)} \right).
\]
The right hand side is majorised by a sum of three terms:
\begin{eqnarray*}
\mbox{r.h.s.} &\le&
\hmu_{\pi^{-1}(U)}( \pi^{-1}(U) \setminus \hat U_n) \\
&& + \ \hat\mu_{\pi^{-1}(U)}\left( \hat y \in \hat U_n \ : \
r_{\hat U_n}(\hat y)> \frac{t}{\mu(U)}\right) \\
&& + \ \hat\mu_{\pi^{-1}(U)}\left( \hat y \in \hat U_n \ : \
r_{\hat U_n}(\hy) > r_{\pi^{-1}(U)}(\hat y) \right)  \\
&=& I + II + III.
\end{eqnarray*}
We have the estimates
\[
I = \frac{ \hat\mu(\pi^{-1}(U) \setminus \hat U_n) }{ \hat\mu(\pi^{-1}(U))
} = \frac{ \hat\mu(\pi^{-1}(U) \setminus \hat U_n) }{ \mu(U) } \le
1-\alpha_n \to 0.
\]
Next
\[
II \le \alpha_n \hat\mu_{\hat U_n} \left( \hat y \ : \
r_{\hat U_n}(\hat y)
> \frac{t}{\mu(U)} \right)  = \alpha_n \hat\mu_{\hat U_n}
\left( \hat y \ : \ r_{\hat U_n}(\hat y) > \frac{\tilde
t}{\hat\mu(\hat U_n)} \right)
\]
for $\tilde t = t \alpha_n$. Theorem~\ref{thm:ret map same ret
stats} says that the return time statistics of a first return map coincides with the return time statistics of the original system. In this case, it means that the system $(\hat I,\hat f, \hat\mu)$ has the same return time statistics on $\hat U_n$ as the induced system $(\hat J_n, \hat F_n, \hat\mu_{\hat J_n})$.  By Lemma~\ref{lem:extendible}, tends to the same return time statistics as $(J_n,F_n,\mu_{F_n})$. Hence $II$ tends  to $\alpha_n G(\tilde t)$ as $\mu(U) \to 0$, and then, by continuity of $G$, to $G(t)$ as $n \to \infty$. The third term
\begin{eqnarray*}
III &=& \hat\mu_{\pi^{-1}(U)}\left[\hat R_U^{-1}(\pi^{-1}(U) \setminus
\hat U_n)\cap \hat U_n\right] \\
&\le& \hat\mu_{\pi^{-1}(U)}(\pi^{-1}(U) \setminus \hat U_n) = I \to 0,
\end{eqnarray*}
as $n \to \infty$. This gives the required upper bound for
$\mu_U(\{ y\ : \ r_U(y) > \frac{t}{\mu(U)} \})$. Now for the lower
bound
\begin{eqnarray*}
\mbox{r.h.s.} &\ge& \hat\mu_{\pi^{-1}(U)}\left( \hat y \in \hat U_n \ : \ r_{\hat U_n}(\hat y) > \frac{t}{\mu(U)}\right)
\\
&& - \hat\mu_{\pi^{-1}(U)}\left( \hat y \in \hat U_n \ : \ r_{\hat U_n}(\hat y) >
r_{\pi^{-1}(U)}(\hat y) \right)  \\
&=& II - III.
\end{eqnarray*}
The above arguments show that this also tends to $G(t)$ as
$\mu(U) \to 0$ and $n \to \infty$. This finishes the proof.
\end{proof}

\section{Exponential return time statistics}
\label{sec:exp stats}

\begin{defn}[Rychlik map]\label{def:Rychlik}
Let $F:\cup_{i\in \N} Y_i \to Y$ be continuous on each $Y_i$,
with $m(\cup_i Y_i)=m(Y) = 1$
for a given reference measure $m$. We call $F$ a \emph{Rychlik
map}, see  \cite{Rychlik}, if:
\begin{enumerate}
\item
there exists a neighbourhood $Z \supset Y$ and for each $i$ a neighbourhood $Z_i \supset Y_i$ such that $F|_{Y_i}$ can be extended to a homeomorphism
between intervals:
$F:Z_i \stackrel{\mbox{\tiny onto}}{\longrightarrow} F(Z_i)$

\item
there exists a function $\Phi:Y\to [-\infty,\infty)$, with ${\rm Var}\
e^\Phi<+\infty$, $\Phi=-\infty$ on $Y \setminus \cup_i Y_i$,
such that the operator $\TF:L^1(m)\to L^1(m)$ defined by
\[
\TF \psi(x)=\sum_{y\in F^{-1}(x)} e^{\Phi(y)} \psi(y)
\]
preserves $m$. In other words, $m(\TF \psi)=m(\psi)$ for each $\psi \in
L^1(m)$ (or equivalently: $m$ is $\Phi$-conformal);
\item
$F$ is expanding: $\displaystyle \sup_{x\in X} \Phi(x) < 0$.
\end{enumerate}
\end{defn}

The following is Theorem 3.2 of \cite{BSTV}.  It also applies to cylinders.

\begin{theorem}
Suppose $(Y,F)$ is a Rychlik map with conformal measure $m$ and
invariant mixing measure $\mu \ll m$.  Then $(Y,F)$ has
exponential return time statistics to balls. \label{thm:exp
for R}
\end{theorem}

\begin{proof}[Proof of Theorem~\ref{thm:exp stats}]
Recall that $F_n$ is the induced map associated to the first return map to the
set $\hat J_n$ in the Hofbauer tower.
As in \cite{BTeqnat}, $F_n$ is a Rychlik map, with induced potential
$\Phi_n=-\delta\log|DF_n|- P(-\delta\log|DF_n|)\tau_{J_n}$.
The conformal measure $m_{\Phi_n}$ is constructed from the induced version of $m_{\delta}$.  Note that the expansivity property (3) follows since for all large $n$, $\sup|DF_n|>1$, and also
$P(-\delta\log|DF_n|)\ge 0$ for $\delta \in [0,1]$, as it is a decreasing function in $\delta$ and
$P(-\log|DF_n|)\ge 0$, see for example \cite{BTeqnat}.
We denote the equilibrium state for the inducing scheme by $\mu_{\delta,F_n}$.  So by Theorem~\ref{thm:exp for R}, each $(J_n,F_n,\mu_{\delta, F_n})$ has exponential return time statistics (\ie $G(t)=e^{-t}$).  Thus Theorem~\ref{thm:induced same as original} implies that $(I,f,\mu_\delta)$ also has exponential return time statistics.
\end{proof}

\begin{proof}[Proof of Corollary~\ref{cor:applications}]
By Theorem~\ref{thm:exp stats}, we only need to guarantee the existence of equilibrium and conformal measures.
The existence of the acip was proved in \cite{BR-LSS}.  The existence of the equilibrium states for $\delta\neq 1$ was proved in \cite{BTeqnat}.  In fact, in that paper we only proved the existence of the relevant conformal measures for inducing schemes.  However, as can be seen in the proof of Theorem~\ref{thm:exp stats}, that is all that is necessary to get exponential return time statistics.
\end{proof}

\begin{proof}[Proof of Proposition~\ref{prop:exp stats cyl}]
The existence of the equilibrium state for $(I, f, \phi)$
was proved in \cite[Theorem 3.4]{kellholder}.  Moreover, it is shown that
the Perron-Frobenius operator with respect to $\phi$-conformal
measure $m_{\phi}$:
\[
\TF:BV_{1,1/p} \to BV_{1,1/p}, \qquad
\TF\psi(x) = \sum_{y \in f^{-1}(x)} e^{\phi(y)} \psi(y)
\]
is quasi-conformal on the space of functions with bounded $p$-variation.
This space includes indicator functions on balls.
By the proof of \cite[Theorem 3.2]{BSTV}, which uses ideas of \cite{HSV},
these facts are sufficient to give exponential return time statistics to balls.
\end{proof}

\section{The Polynomial Gibbs Property}
\label{sec:W Gibbs}

We prove Theorem~\ref{thm:weak G} in two parts.  The case of equilibrium states for potentials in $\H$ is treated in Proposition~\ref{prop:poly Gibbs},
and then the upper and lower bounds for the acip is separated into
two lemmas.

Proposition~\ref{prop:poly Gibbs} relies on the fact that the
equilibrium states $\mu=\mu_\phi$ obtained in \cite{BTeqgen}
have exponential tails for an induced system, and also that $\phi\in \H$ are bounded.

\begin{prop}
There exists $\kappa\in (0,\infty)$ \st for $\mu$-a.e. $x$ there exists
$n_0=n_0(x)\in \N$ \st $n\ge n_0$ implies
$$
\frac 1{n^\kappa} \le \frac{\mu(\c_n[x])}{e^{S_n\phi(x)-nP(\phi)}} \le n^\kappa
$$
where $S_n \phi(x):=\phi\circ f^{n-1}(x)+\cdots + \phi(x)$.
\label{prop:poly Gibbs}
\end{prop}

This result can be compared with Lemmas 3.2 and 3.3 of \cite{Pac}.

\begin{proof}
Here we use results from \cite{BTeqgen}, which are based on a slightly different type of inducing scheme to that in the rest of this paper.
So let us briefly explain these inducing schemes.
Let $\hat Y$ be an interval compactly contained in some domain $D\in \D$,
and such that $Y:=\pi(\hat Y)\in \P_n$ for some $n$.  Then for $x \in Y$, let $\tau(x)$ be the first return time of the point $\hat x:=\pi^{-1}(x)\cap \hat Y$ to $\hat Y$.  As explained in \cite{BTeqgen} (see also the proof of \cite[Theorem 3]{BTeqnat}), this gives an inducing scheme $(Y, F)$ with bounded distortion, where $F=f^\tau$.

For $x\in Y$ let $\Phi(x):= S_{\tau(x)} \phi(x) =
\phi\circ f^{\tau(x)-1}(x)+\cdots + \phi(x)$, and
$S_n \Phi(x) := \Phi \circ F^{n-1}(x)+\cdots + \Phi(x)$.
We denote the measure on the inducing scheme by $\mu_\Phi$.

Let $T_\phi$ and $T_\Phi$ denote the set of typical points of $\mu_\phi$ and $\mu_\Phi$ respectively.  Let $f^{k_1}(x)=y_1$ be the first time that $x$ maps into $T_\Phi$, and let $k_2\in \N$ be minimal \st there exists $y_2\in T_\Phi$ with $f^{k_2}(y_2)=x$.
For $n\ge \max\{k_1+\tau(x),k_2+\tau(y_2)\}$,
\begin{align*}
\mu_\phi(\c_n[x])& \le \mu_\phi(f^{-k_1}(\c_{n-k_1}[y_1])) = \mu_\phi(\c_{n-k_1}[y_1])\\
\mu_\phi(\c_n[x])&=\mu_\phi(f^{-k_2}(\c_{n}[x])) \ge \mu_\phi(\c_{n-k_2}[y_2]).
\end{align*}
Therefore, we may assume that $x\in T_\Phi$.

By the Gibbs property for $(Y,F,\mu_\Phi)$, there exists $K>0$ \st
$$
\frac1K\le \frac{\mu_\Phi(\c_{\tau^n(x)}[x])}{e^{S_n\Phi(x)}}\le K.
$$
We will use the fact that there exists $\rho(x)\in (0,\infty)$ \st for a
nested sequence of open sets $\{U_n\}_n$ \st $\cap_n U_n = \{ x \}$ as $n\to \infty$, we have $\frac{\mu_\Phi(U_n(x))}{\mu_\phi(U_n(x))} \to \rho(x)$.  Thus, for large enough $n$, the estimates we need for $\mu_\phi(\c_n[x])$ follow immediately from those for $\mu_\Phi(\c_n[x])$.

For each large $n$, there exists $k$ \st $\tau^{k-1}(x)< n\le \tau^k(x)$.  We get
$$
\frac{\mu_\Phi(\c_{\tau^k(x)}[x])}{e^{S_n\phi(x)} }\le  \frac{\mu_\Phi(\c_n[x])}{e^{S_n\phi(x)}} \le \frac{\mu_\Phi(\c_{\tau^{k-1}(x)}[x])}{e^{S_n\phi(x)}}.
$$
So the Gibbs property implies
$$
\frac{e^{S_{\tau^k(x)-n}\phi(f^n(x))}}{K}\le
\frac{\mu_\Phi(\c_n[x])}{e^{S_n\phi(x)}} \le
Ke^{-S_{n-\tau^{k-1}\phi(x)}(F^{k-1}(x))}.
$$
Since $|S_{\tau^k(x)-n}\phi(f^n(x))| \le \sup_{x\in I} |\phi(x)||\tau^k(x)-\tau^{k-1}(x)|$,
it is sufficient for the lower bound to show that $\tau(F^k(x)) \le \kappa\log n$ for all large $n$.
\begin{claim}
There exists $\kappa\in (0,\infty)$ \st for $\mu_\Phi$-a.e. $x\in Y$ there exists $k_0=k_0(x)\in \N$ \st $k\ge k_0$ implies $\tau(F^k(x))\le \kappa\log k.$
\label{claim:gaps}
\end{claim}
\begin{proof}
We use the fact that $(Y,F,\mu_\Phi)$ has exponential tails: in \cite{BTeqgen} it is shown that
there exists $\alpha>0$ such that $\mu_\Phi\{\tau\ge k\} \le Ce^{-\alpha k}$.
We fix $\kappa>\frac1\alpha$. Let $V_k:=\{x\in Y: \tau(F^k(x))> \kappa\log k\}$.  Since $\mu_\Phi$ is $F$-invariant and $V_k= F^{-k}\{\tau\ge \kappa\log k\}$, we have $\mu_\Phi(V_k) \le Cn^{-\alpha\kappa}$.
So by the \bcl we know that for $\mu_\Phi$-a.e. $x$ there exists $k_0=k_0(x)$ \st $k\ge k_0$ implies $x\notin V_k$.
\end{proof}
From this claim it follows that $\tau(F^k(x)) \le \kappa\log k$ for all large $k$.
Hence $\tau(F^k(x)) < \kappa\log n$ for all large $n$.  The upper bound follows similarly.
\end{proof}

We now show the polynomial Gibbs property for acips.
For $N\in \N$, $\ell >1$ and $K>0$, let $\A(N,\ell,K)$
be the set of maps in $NF^3$ with $\#\crit=N$ and with each critical point
$c\in \crit$ having order $\ell_c<\ell$ and satisfying
$$
|Df^n(f(c))|\ge K \hbox{ for all sufficiently large } n.
$$
Clearly, whenever $\min_{c \in \scrit} |Df^n(f(c))| \to \infty$ as
$n \to \infty$, $f$ it must lie in $\A(N,\ell,K)$ for some $N,\ \ell$
and any $K>0$.

We let $m$ denote Lebesgue measure on the interval $I$.
The following is proved in \cite[Proposition 4]{BR-LSS}.

\begin{prop}
Let $\ell>1$ and $N\in \N$.  There exists $K>0$ \st if $f\in \A(N,\ell, K)$
then there is $C > 0$ \st  for any Borel set $A$ and any $n \ge 0$,
$$
m(f^{-n}(A))\le Cm(f(A))^{\frac{1}{2\ell_{max}}}.
$$
\label{prop:multiacip}
\end{prop}

We can construct the invariant measure $\mu$ by taking a limit of the Ces\`aro means $\frac 1n\sum_{k=0}^{n-1}m\circ f^{-k}$.  From
Proposition~\ref{prop:multiacip}, it is easy to see that for an
$m$-measurable set $A$, we have \begin{equation} \mu(A) \le Cm(f(A))^{\frac{1}{2\ell_{max}}}\le Cm(A)^{\frac{1}{2\ell_{max}^2}}.\label{eq:ac}\end{equation}
In particular, $\mu \ll m$.

We prove Theorem~\ref{thm:weak G} for acips in two lemmas.  First the upper and then the lower bound.

\begin{lemma}[Upper bound] Fix $\gamma'>2$.  For $\mu$-a.e. $x$ there is
$n_0=n_0(x)$ such that for all $n \ge n_0$
$$\mu(\c_n[x])|Df^n(x)|\le n^{\gamma'}.$$
\end{lemma}

\begin{proof}
Our proof follows \cite[Lemma 5]{BrV}.  We will use the \bcl applied to $m$ repeatedly here. Let $\gamma'':=\frac{\gamma'}2>1$. Let $W_n:=\{\c_n\in \P_n: \mu(\c_n)>n^{\gamma''}m(\c_n)\}$, and $A_n:=\bigcup_{\c_n\in W_n} \c_n$. Since $\mu$ is a probability
measure, we have $$1 \ge \mu(A_n)= \sum_{\c_n \in W_n}\mu(\c_n) \ge n^{\gamma''}\sum_{\c_n \in
W_n}m(\c_n) =n^{\gamma''}m(A_n).$$  Whence $m(A_n) \le n^{-\gamma''}$.  The \bcl implies
that $m$-a.e. $x\in I$ belongs to $A_n$ for only finitely many $n$.

Now for any $\c_n\in \P_n$, let $$U(\c_n)= \left\{x\in \c_n:|Df^n(x)|> \frac{n^{\gamma''}}{m(\c_n[x])}\right\}.$$  Then for $\c_n\in \P_n$,
$$
1 \ge m(f^n(\c_n))\ge  \int_{U(\c_n)}|Df^n(x)| dm \ge \frac{n^{\gamma''}}{m(\c_n)} m(U(\c_n)),
$$
so $m(\c_n)\ge n^{\gamma''}m(U(\c_n[x]))$. Letting $B_n := \bigcup_{\c_n\in \P_n}U(\c_n)$, we have $$m(B_n) = \sum_{\c_n\in \P_n}m(U(\c_n))\le n^{-\gamma''}\sum_{\c_n\in \P_n}m(\c_n) \le n^{-\gamma''}.$$  So again the \bcl implies that $m$-a.e. $x$ belongs to $B_n$ for only finitely many $n$.

Therefore since $\mu\ll m$, for $\mu$-a.e. $x\in I$ there exists some $n_0=n_0(x)$ \st $x\notin A_n\cup B_n$ for all $n \ge n_0$.  Thus $n\ge n_0$ implies
$$
\mu(\c_n[x])|Df^n(x)| \le n^{\gamma''}m(\c_n[x])
\left( \frac{n^{\gamma''}}{m(\c_n[x])}\right) = n^{\gamma'}
$$
and we have the required upper bound.
\end{proof}

Notice that, unlike the following lemma, the proof of the above lemma did not require Proposition~\ref{prop:multiacip}.

\begin{lemma}[Lower bounds]
For $\mu$-a.e. $x$ there is $n_0$ such that for
all $n \ge n_0$, and $\mu$ an acip,
$$
\frac1{n^{2\gamma}} \le \frac{|f^n(\c_n[x])|}{n^{\gamma}}
\le \mu(\c_n[x])|Df^n(x)|.
$$
\end{lemma}

\begin{proof}
Let
$$
V_n:=\Big\{x\in I:|f^n(x)-\bd f^n(\c_n[x])|<n^{-\gamma}|f^n(\c_n[x])|\Big\}.
$$
For $x\in I$, denote the part of $f^n(\c_n[x])$ which lies within $n^{-\gamma}|f^n(\c_n[x])|$ of the boundary of $f^n(\c_n[x])$ by $E_n[x]$.  We will estimate the Lebesgue measure of the pullback $f^{-n}(E_n[x])$.  Note that this set consists of more than just the pair of connected components $\c_n[x]\cap V_n$.

Clearly, $m(E_n[x]) \le 2 n^{-\gamma}m(f^n(\c_n[x]))$.  Hence from \eqref{eq:ac}, which follows from Proposition~\ref{prop:multiacip}, we have
$$
m(V_n\cap f^{-n}(E_n[x]))
\le K_0(2n^{-\gamma}m(f^n(\c_n[x])))^{\frac1{2\ell_{max}^2}}
\le 2K_0 n^{-\frac{\gamma}{2\ell_{max}^2}}.
$$
There are at most $2n\# \crit$ domains $f^n(\c_n[x])$, hence $$m(V_n)\le Cn^{1-\frac{\gamma}{2\ell_{max}^2}}.$$  For $\gamma>4\ell_{max}^2$ we have $\sum_n m(V_n)<\infty$.  So by the \bcl for $m$-a.e. $x$ there exists $n_0$ \st $x\notin V_n$ for $n\ge n_0$.

We fix $0<\delta<1$ and may assume that $n_0^{-\gamma}<\delta$.  Let $\tilde \c_n[x]\subset \c_n[x]$ be the maximal interval for which $d(f^n(\tilde \c_n[x]),\bd f^n(\c_n[x]))=\frac\delta2|f^n(\c_n[x])|$.  Then for $x$ as above, by the Koebe Lemma we obtain for $n\ge n_0$,
$$
|Df^n(x)| \ge \left( \frac{n^{-\gamma}}{1+n^{-\gamma}}\right)^2\frac{|f^n(\tilde \c_n[x])|}{|\tilde \c_n[x]|} \ge
\left(\frac{1-\delta}{2n^\gamma}\right) \ \frac{|f^n(\c_n[x])|}{|\c_n[x]|}.
$$
Letting $b:=\inf_{x\in \mbox{\tiny supp}(\mu)}\frac{d\mu}{dm}(x)$,
we have
$$\mu(\c_n[x])|Df^n(x)| \ge b\left(\frac{1-\delta}{2n^\gamma}\right)|f^n(\c_n[x])|$$
and the first part of the proof is finished if we can show that
$b>0$. Notice that since this works for $m$-a.e. $x$, it must
also work for $\mu$-a.e. $x$.  To understand why $b>0$, first note
that by the Folklore Theorem, see \cite{MSbook}, the invariant measure $\mu_F$ for the induced system
$(Y,F)$ has $b'>0$ so that $\frac{d\mu_F}{dm}\ge b'$ on $\supp(\mu_F)$.
Also, there exists $N$ such that
$\supp(\mu) \subset \overline{ \cup_{k=0}^N f^k(Y)}$.
Given $y \in \supp(\mu)$ and a
set $U\subset Y$ so that $y \in f^k(U)$ for $k \le N$,
we have
$$
\frac{\mu(f^k(U))}{m(f^k(U))} \ge \frac{\mu(f^k(U))}{\int_{U}|Df^k|~dm}
\ge \frac{\mu(U)}{m(U) (\sup|Df|)^k} \ge \frac{b'}{(\sup|Df|)^k}.
$$
Then shrinking $U$ we see that $\frac{d\mu}{dm}(y)>b$
where $b:= \frac{b'}{(\sup|Df|)^N}$.

Let $W_n:=\{x\in I:|f^n(\c_n[x])|\le n^{-\gamma}\}$.  For each domain $Z_n$ of $W_n$, we choose a point $x_k\in Z_n$, so that $W_n=\cup_{k=1}^{p_n}Z_n[x_k]$.  We have

$$
m(W_n) = m\left(\bigcup_{k=1}^{p_n}Z_n[x_k]\right) \le   m\left(\bigcup_{k=1}^{p_n}f^{-n}[f^n(Z_n[x_k])]\right) \le (2n\#\crit) n^{-\frac\gamma{2\ell_{max}^2}}.
$$
Since $\gamma>4\ell_{max}^2$, the \bcl implies that for $\mu$-a.e. $x\in I$ there
is some $n_0\ge 1$ \st $n\ge n_0$ implies $m(f^n(\c_n[x])) \ge n^{-\gamma}$.  Combining this lower bound with the one above, we are finished.
\end{proof}

\section{Entropy fluctuations}
\label{sec:fluc}

In this section we prove Theorem~\ref{thm:fluc}.  This follows the same path as the proof of Theorem 3 in \cite{BrV}.
In the case that $\mu$ is an acip, a sketch of the proof is as follows:

{\bf Step 1:} The log-normal fluctuations in the Ornstein-Weiss Theorem
follow (using \cite{Sau}) from
\\
(i) exponential return time statistics to cylinders (which is true for our equilibrium states by Proposition~\ref{prop:exp stats cyl}, and for acips by Theorem~\ref{thm:exp stats} applied to cylinders); and
\\
(ii) log-normal fluctuations in the Shannon-McMillan-Breimann Theorem.

{\bf Step 2:} Condition (ii) reduces to the usual Central Limit Theorem for
the observable $\phi = \log |Df| - \int \log|Df| \ d\mu$, provided
there is $\alpha < \frac12$ such that
\[
\frac1{n^\alpha}\le \Big| \log(\ \mu(\c_n[x]) |Df^n(x)|\ ) \Big| \le n^\alpha
\]
for $\mu$-a.e. $x$ and $n$ sufficiently large.
Our polynomial Gibbs property clearly implies this.

{\bf Step 3:} To prove the CLT for $\phi$, we need Gordin's Theorem
(see \cite[Theorem 6]{BrV}), for which we need to verify that
$\phi \in L^2(\mu)$.  Let us do that here.

\begin{lemma}
The potential $\phi:=\log|Df|-\int\log|Df|~d\mu$ belongs to $L^2(\mu)$.
\label{lem:potential}
\end{lemma}

\begin{proof}
Clearly it is enough to show that $\log|Df|\in L^2(\mu)$.  Clearly,
there exists some $C>0$ \st $\log|Df(x)| \le C\ell_{max}\log|x-\crit|$.
Also by construction of $\mu$ and Proposition~\ref{prop:multiacip}, we have $\mu(B_\eps(c))\le
C\eps^{\frac1{2\ell_{max}^2}}$ for any $c\in \crit$.  For a given $c\in \crit$, let $U$ be a neighbourhood of $c$ which is away from any other element of $\crit$.  We have
\begin{align*}
\int_U(\log|Df(x)|)^2~d\mu & \le
2 \sum_n \int_{(c+2^{-(n+1)},c+2^{-n})}(\log|Df|)^2~d\mu\\
& \le 2C \sum_n 2^{\frac{1-n}{2\ell_{max}^2}}|C\ell_{max}\log 2^{-n}|^2\\
& \le 4C^3 \ell_{max}^2(\log 2)^2\sum_n n^2 2^{-\frac{n}{2\ell_{max}^2}}
<\infty.
\end{align*}
Since we can perform such a calculation at every critical point,
the lemma is proved.
\end{proof}

{\bf Step 4:} Finally, to apply Gordin's Theorem, we follow pages 91-93 of \cite{BrV} verbatim, except that neighbourhoods $B(c,L^{-n})$ and $B(c,n^{-5})$ of the critical point $c$ need to be replaced by neighbourhoods $B(\crit,L^{-n})$ and $B(\crit,n^{-5})$ of $\crit$.
The argument in \cite[page 93]{BrV} that
$\int_\Delta |P^n(\tilde \phi \tilde h)| d\tilde m$ decays sufficiently fast
can be done in the multimodal case too, using \cite{BrLS}
and finally using \cite{BR-LSS} to remove the assumption
from  \cite{BrLS} that all critical points have the same order.  (See the use of \cite[Lemma 9]{BTeqnat} for an application of this method.)

For $\mu$ an equilibrium state for a potential $\phi \in \H$, the proof is simplified.  Step 1 is the same, so we only need to know that $\mu$ satisfies the weak Gibbs property, coupled with the fact that $\phi$ satisfies the CLT for $(I,f,\mu)$.  The first fact follows from Proposition~\ref{prop:exp stats cyl}, and the second follows from \cite[Theorem 3.3]{kellholder}.

\medskip
\noindent
Department of Mathematics\\
University of Surrey\\
Guildford, Surrey, GU2 7XH\\
UK\\
\texttt{h.bruin@surrey.ac.uk}\\
\texttt{http://www.maths.surrey.ac.uk/}

\medskip
\noindent
Department of Mathematics\\
University of Surrey\\
Guildford, Surrey, GU2 7XH\\
UK\footnote{
{\bf Current address:}\\
Departamento de Matem\'atica Pura\\
Faculdade de Ci\^encias da Universidade do Porto\\
Rua do Campo Alegre, 687\\
4169-007 Porto\\
Portugal\\
}\\
\texttt{mtodd@fc.up.pt}\\
\texttt{http://www.fc.up.pt/pessoas/mtodd/}

\end{document}